\documentclass[conference]{ieeeconf}
\IEEEoverridecommandlockouts
\usepackage{epsfig} 
\usepackage{times} 
\usepackage{amsmath} 
\usepackage{amssymb}  
\usepackage{url}
\usepackage{enumerate}
\usepackage{changes}
\usepackage{cite}
\usepackage{color}
\usepackage{xspace}
\usepackage{algorithmic}
\usepackage[ruled]{algorithm}
\graphicspath{{./figures/}}
\usepackage{graphicx}
\usepackage{epstopdf}
\usepackage{float}
\DeclareGraphicsExtensions{.png}
\usepackage{comment}
\usepackage{caption}

\usepackage{subcaption}

\newtheorem{lemma}{Lemma}
\newtheorem{theorem}{Theorem}

\newtheorem{assumption}{Assumption}

\title{\LARGE \bf A New Family of Feasible Methods for Distributed Resource Allocation}

\author{Xuyang Wu, Sindri Magn\'{u}sson, and Mikael Johansson\thanks{X. Wu and M. Johansson are with the Division of Decision and Control Systems, School of Electrical Engineering and Computer Science, KTH Royal Institute of Technology, SE-100 44 Stockholm, Sweden. Email: {\tt {\{xuyangw,mikaelj\}@kth.se}.}}
\thanks{S. Magn\'{u}sson is with the Department of Computer and System Science, Stockholm University, SE-164 07 Stockholm, Sweden. Email: {\tt sindri.magnusson@dsv.su.se}.}
\thanks{This work was supported in part by the funding from Digital Futures and in part by the Swedish Research Council (Vetenskapsr\r{a}det) under grant 2020-03607.}}

\begin{document}

	\maketitle
	
	\begin{abstract}
	Distributed resource allocation is a central task in network systems such as smart grids, water distribution networks, and urban transportation systems. When solving such problems in practice it is often important to have non-asymptotic feasibility guarantees for the iterates, since over-allocation of resources easily causes systems to break down. In this paper, we develop a distributed resource reallocation algorithm where every iteration produces a feasible allocation. The algorithm is fully distributed in the sense that nodes communicate only with neighbors over a given communication network. We prove that under mild conditions the algorithm converges to a point arbitrarily close to the optimal resource allocation. Numerical experiments demonstrate the competitive practical performance of the algorithm.
	\end{abstract}

	%
	%
\section{Introduction}

This paper develops distributed mechanisms for resource allocation in a network of cooperating agents, with applications in smart grids, water distribution networks, and urban transportation systems.
Although a large number of distributed resource allocation algorithms have been proposed~\cite{Aybat16,Nedic17,XiaoL06b,Necoara13,Ghadimi13,Ho80,Magnusson16,Magnusson18}, most of them do not guarantee feasibility of iterates~\cite{Aybat16,Nedic17} or do not allow for local constraints~\cite{Ghadimi13,XiaoL06b,Necoara13}. However, many practical problems, such as economic dispatch, have local constraints and require that iterates are feasible at all times to avoid system breakdown.

Feasible methods for distributed resource allocation that allow for local constraints include \cite{Ho80,Magnusson16,Magnusson18}, among which \cite{Ho80} addresses the problem with one global linear equality constraint while \cite{Magnusson16,Magnusson18} consider the problem with one global linear inequality constraint. In addition, \cite{Ho80,Magnusson16,Magnusson18} can only deal with problems with one-dimensional local decision variables, \cite{Ho80} requires the local constraints to be the set of non-negative real numbers, and \cite{Magnusson16,Magnusson18} rely on star networks.

Motivated by the lack of distributed feasible methods that can solve a broader range of optimal resource allocation problems, we design a distributed resource reallocation algorithm (DRRA), which allows \emph{multiple global inequality and equality coupling constraints and local constraints on multi-dimensional local variables, still guarantees that the iterates are feasible at all times.} In addition, DRRA can be implemented on general undirected, connected networks, while \cite{Magnusson16,Magnusson18} can only handle star networks.

The outline of this paper is as follows: Section \ref{sec:probformuandtrans} formulates the problem and introduces a problem transformation, which facilitates the development of DRRA detailed in Section \ref{sec:Algorithm}. Section \ref{sec:convergence} analyses the convergence properties of DRRA and Section \ref{sec:numericalexample} evaluates its practical performance in simulations. Finally, Section \ref{sec:conclusion} concludes the paper.
\subsection*{Notation}

For any set $\mathcal{S}\subseteq \mathbb{R}^n$, $|\mathcal{S}|$ represents its cardinality. We use $\log$ to denote the natural logarithm and $\|\cdot\|$ the $\ell_2$ vector norm. We say that a function $f:\mathbb{R}^d\rightarrow \mathbb{R}\cup\{+\infty\}$ is smooth on a set $X\subseteq \mathbb{R}^d$ if it is differentiable on $X$ and its gradient $\nabla f$ is Lipschitz continuous on $X$, i.e., there exists $L\ge 0$ such that $\|\nabla f(x)-\nabla f(y)\|\le L\|x-y\|$ $\forall x,y\in X$. In addition, $\mathbf{0}$ is the all-zero vector with proper dimension.

\section{Problem Formulation and Transformation}\label{sec:probformuandtrans}
This section formulates the problem, introduces a problem transformation for algorithm development in Section \ref{sec:Algorithm}, and provides two motivating examples.

\subsection{Problem Formulation}

Consider a network of nodes $\mathcal{V}=\{1,\cdots,n\}$ that communicate over an undirected connected graph $\mathcal{G}=(\mathcal{V}, \mathcal{E})$ induced by the edge set $\mathcal{E}\subseteq \{\{i,j\}:i,j\in\mathcal{V}, i\ne j\}$. The goal of the agents is to communicate with their network  neighbors to find the optimal resource allocation:
\begin{equation}\label{eq:prob}
\begin{split}
\underset{x_i\in\mathbb{R}^{d_i},~i\in\mathcal{V}}{\operatorname{minimize}}~&\sum_{i\in\mathcal{V}} f_i(x_i)\\
\operatorname{subject~to}~&\sum_{i\in\mathcal{V}} A_i^{\text{in}}x_i\le b^{\text{in}},\\
&\sum_{i\in\mathcal{V}} A_i^{\text{eq}}x_i = b^{\text{eq}},\\
& x_i \in \mathcal{X}_i,~~\forall i\in \mathcal{V},
\end{split}
\end{equation}
where $f_i:\mathbb{R}^{d_i}\rightarrow\mathbb{R}$, $A_i^{\text{in}}\in\mathbb{R}^{m^{\text{in}}\times d_i}$, $A_i^{\text{eq}}\in\mathbb{R}^{m^{\text{eq}}\times d_i}$, and $\mathcal{X}_i\subseteq \mathbb{R}^{d_i}$. Moreover, $b^{\text{in}}\in\mathbb{R}^{m^{\text{in}}}$ and $b^{\text{eq}}\in\mathbb{R}^{m^\text{eq}}$, where $m^{\text{in}}$ and $m^{\text{eq}}$ are the total number of inequality and equality constraints, respectively.
%
%
For our theoretical analyses we will impose the following assumption.
\begin{assumption}\label{asm:prob}
	The following holds for problem \eqref{eq:prob}.
	\begin{enumerate}[(a)]
		\item Each $f_i$ is convex and twice-continuously differentiable on $\mathcal{X}_i$.
		Moreover, we can write the local constraint as 
		$$\mathcal{X}_i=\{x\in \mathbb{R}^{d_i}:~g_i^j(x)\leq 0, j=1,\cdots, p_i\}, $$
		where $g_i^j: \mathbb{R}^{d_i}\rightarrow \mathbb{R}$ is convex and twice-continuously differentiable for $j=1,\ldots,p_i$.
		\item There exist $\tilde{x}_i\in \mathbb{R}^{d_i}$ $\forall i\in\mathcal{V}$ such that $g_i^{j}(\tilde{x}_i)<0$ for all $i\in \mathcal{V}$  and $j=1,\ldots,p_i$, and 
		$\sum_{i\in\mathcal{V}} A_i^{\text{in}}x_i\le b^{\text{in}}$  and $\sum_{i\in\mathcal{V}} A_i^{\text{eq}}x_i = b^{\text{eq}}$.

		\item The feasible set of problem \eqref{eq:prob} is compact.
	
		\item Each $A_i:=[(A_i^{\text{in}})^T, (A_i^{\text{eq}})^T]^T\in\mathbb{R}^{(m^{\text{in}}+m^{\text{eq}})\times d_i}$ has full row rank, i.e., $\operatorname{rank}(A_i)=m^{\text{in}}+m^{\text{eq}}$ $\forall i\in\mathcal{V}$.

	\end{enumerate}
\end{assumption}
In Assumption \ref{asm:prob}, conditions (a)--(b) are standard in the literature and allow us to study the convergence of our algorithm using the KKT conditions. In particular, condition (b) is similar to Slater's condition, commonly used in constrained convex optimization. Condition (c) is satisfied if either $\mathcal{X}_i$ is compact for all $i\in \mathcal{V}$ or the combination of $\mathcal{X}_i$ and the coupling constraints define a compact set.  For example, if 
	\begin{equation}\label{eq:examplecompactconstraints}
	\sum_{i\in\mathcal{V}} x_i\le b,~x_i\ge \mathbf{0},~\forall i\in\mathcal{V}.
	\end{equation}
Condition (d) is needed in our analysis to ensure the convergence of the proposed algorithm. 

\subsection{Problem Transformation Using Barrier Function}\label{ssec:problembarriertran}

 The main challenge in our algorithm development lies in ensuring the feasibility of the coupling constraints at every iteration.   To deal with the local constraint sets $\mathcal{X}_i$, we  use barrier functions \cite{Bertsekas99}. To that end, we consider the following transformed version of problem \eqref{eq:prob}:
%
%
\begin{equation}\label{eq:barrierprob}
\begin{split}
\underset{x_i\in \tilde{\mathcal{X}}_i,i\in\mathcal{V}}{\operatorname{minimize}}~~&\sum_{i\in\mathcal{V}} F_i(x_i):=(f_i(x_i)+c\sum_{j=1}^{p_i}B_i^j(x_i))\\
\operatorname{subject~to}~&\sum_{i\in\mathcal{V}} A_i^{\text{in}}x_i\le b^{\text{in}},\\
&\sum_{i\in\mathcal{V}} A_i^{\text{eq}}x_i = b^{\text{eq}},
\end{split}
\end{equation}
where $c>0$, $\tilde{\mathcal{X}}_i=\{x\in\mathbb{R}^{d_i}:~g_i^j(x)<0, j=1,\ldots,p_i\}$ is the interior of $\mathcal{X}_i$, and $B_i^j(x_i):\{x\in\mathbb{R}^{d_i}:g_i^j(x)<0\}\rightarrow\mathbb{R}$ is a barrier function on the constraint $g_i^j(x_i)\le 0$. We require each $B_i^j$, $i\in\mathcal{V}$, $j=1,\ldots,p_i$ satisfy the following conditions under Assumption \ref{asm:prob}:
\begin{enumerate}[(a)]
	\item If $g_i^j(x_i)$ goes to $0$ from negative values, then $B_i^j(x_i)$ goes to $+\infty$.
	\item Each $B_i^j$ is convex and twice-continuously differentiable on $\tilde{\mathcal{X}}_i$.
\end{enumerate}
Clearly, if Assumption \ref{asm:prob} holds, then each $\tilde{\mathcal{X}}_i$ is convex and each $F_i$ is convex and twice continuously differentiable on $\tilde{\mathcal{X}}_i$. Moreover, the solution to problem~\eqref{eq:barrierprob} is an approximate solution to problem~\eqref{eq:prob} that converges to the true solution as $c$ goes to $0$. 
Two standard barrier functions satisfying these  conditions are
	a) the logarithmic barrier function $B_i^j(x_i)=-\log(-g_i^j(x_i))$ and
	b) the inverse barrier function $B_i^j(x_i)=-{1}/{g_i^j(x_i)}$.

\subsection{Motivating Example} \label{sec:ME}
 We now illustrate two problems in smart-grids where feasible methods might be needed. Neither problem can be handled by existing feasible methods. 
\subsubsection{Economic dispatch \cite{Yang13}}\label{sssec:ed}
Consider a smart-grid with several users, where some users can generate power and some users have flexible demands.  
The users wish to find in cooperation the optimal power allocations:
\begin{equation}\label{eq:ED}
\begin{aligned}
& \underset{x_i\in \mathbb{R},i\in\mathcal{V}}{\text{minimize}}
& & \sum_{i\in\mathcal{V}} f_i(x) \\
& \text{subject to} 
& & \sum_{i\in\mathcal{V}} x_i = b,\\
&&& x_i\in \mathcal{X}_i,~~i \in \mathcal{V}
\end{aligned}
\end{equation}
where $x_i\in \mathbb{R}$ is the power injection of node $i\in\mathcal{V}$ (where negative injection $x_i<0$ means that node $i$ consumes power from the grid). The objective functions $f_i(x_i)$ are generation costs or the disutility related to shifting power demands. Here $b\in \mathbb{R}$ is the power injection not accounted for by the nodes in $\mathcal{V}$. The local constraints $\mathcal{X}_i$ are important. They specify hard constraints on devices or user preferences. In practice, $\mathcal{X}_i$ is almost always compact. The coupling constraint ensures that the supply meets the demand, which is a hard physical constraint in power systems. 

Existing distributed feasible methods \cite{XiaoL06b,Necoara13,Ghadimi13,Ho80,Magnusson16,Magnusson18} cannot handle this problem. The work in \cite{XiaoL06b,Necoara13,Ghadimi13} cannot handle local constraints and \cite{Ho80} requires $\mathcal{X}_i$ to be the set of non-negative real numbers. These works cannot handle compact local constraints by introducing barrier penalty function in $f_i(\cdot)$, since they require  $f_i(\cdot)$ to be smooth on $\mathcal{X}_i$. On the other hand, we do not require smoothness on $f_i(\cdot)$ and can thus handle local constraints using barrier functions. Moreover, the development in \cite{XiaoL06b,Necoara13,Ghadimi13,Ho80} relies on the fact that the coupling constraint holds with equality, whereas we can easily deal with both equality and inequality coupling constraints.
The works in~\cite{Magnusson16, Magnusson18} require a star communication network, whereas we can deal with general communication topology. Moreover,~\cite{Magnusson16, Magnusson18} cannot handle a coupling constraint that holds with equality.    

\subsubsection{Multiple Resources \cite{Enyioha18}}\label{sssec:edmultiple}
 In smart grids, users might get power from different sources, e.g., renewable or coal. To model such scenarios, we need to consider multiple coupling constraints, e.g., as in the following example:
\begin{equation}\label{eq:problemmultiple}
\begin{aligned}
& \underset{x_i\in \mathbb{R}^2, i\in\mathcal{V}}{\text{minimize}}
& & \sum_{i\in\mathcal{V}} f_i(x_i^{\text{renew}},x_i^{\text{coal}})\\
& \text{subject to}
& & \sum_{i\in\mathcal{V}} x_i^{\text{renew}} = 0,
~~\sum_{i\in\mathcal{V}} x_i^{\text{coal}} = 0,\\
&&& x_i\in \mathcal{X}_i,~i\in\mathcal{V},
\end{aligned}
\end{equation}
where the local objective functions might, e.g., encode user preferences for different energy sources. Existing feasible methods cannot handle this problem or, in general, problems with multiple resources. Our distributed feasible approach (presented in the next section) can solve both problems above.

\section{Algorithm}\label{sec:Algorithm}

This section develops the DRRA algorithm which solves problem \eqref{eq:prob}, based on the transformation in Section \ref{ssec:problembarriertran}.

\subsection{Algorithm Idea: Right-Hand Side Allocation}

 Our algorithm builds on the principle of decomposition by right-hand side allocation \cite{Bertsekas99}. This is a primal decomposition method that splits  problem~\eqref{eq:barrierprob} into multiple smaller problems, one for each node in $\mathcal{V}$. 
 %
 %
 %
 %
%
%
The main idea is to introduce auxiliary variables 
$y_i^{\text{in}}\in \mathbb{R}^{m^{\text{in}}}$ and $y_i^{\text{eq}}\in \mathbb{R}^{m^{\text{eq}}}$, one for each coupling constraint. Problem \eqref{eq:barrierprob} can then be written equivalently as follows: 
\begin{equation}\label{eq:probxyz}
\begin{split}
\underset{x_i\in \tilde{\mathcal{X}}_i,y_i\in\mathbb{R}^m,i\in\mathcal{V}}{\operatorname{minimize}} ~&~~\sum_{i\in \mathcal{V}} F_i(x_i)\\
\operatorname{subject~to}\quad&~~
A_i^{\text{in}}x_i\le y_i^{\text{in}},~\forall i\in\mathcal{V},\\ 
&~~ A_i^{\text{eq}}x_i=y_i^{\text{eq}},~\forall i\in\mathcal{V},\\
&~~ \sum_{i\in\mathcal{V}} y_i = b,
\end{split}
\end{equation}
where $m=m^{\text{in}}+m^{\text{eq}}$, $b=[(b^{\text{in}})^T, (b^{\text{eq}})^T]^T\in\mathbb{R}^m$, and $y_i=[(y_i^{\text{in}})^T,(y_i^{\text{eq}})^T]^T\in\mathbb{R}^m$.
We can express our problem in terms of the primal functions $\phi_i:\mathbb{R}^m\rightarrow\mathbb{R}\cup\{+\infty\}$ for each node $i\in\mathcal{V}$, defined as the optimal value of the following problem
\begin{equation} \label{eq:subproblem}
\begin{aligned}
& \underset{x_i\in \tilde{\mathcal{X}}_i}{\text{minimize}}
& & F_i(x_i) \\
& \text{subject to}
& & A_i^{\text{in}}x_i\le y_i^{\text{in}},\\
&&& A_i^{\text{eq}}x_i=y_i^{\text{eq}},
\end{aligned}
\end{equation}
for a given right-hand side vector $y_i$.
%
%
%
%
%
%
%
Note that problem \eqref{eq:subproblem} is convex for all $i\in\mathcal{V}$, see~\cite[Section 5.6.1]{Boyd04}, and that problem \eqref{eq:probxyz} can be equivalently phrased as the minimization of the sum of the primal functions
\begin{equation}\label{eq:probleminyz}
\begin{split}
\underset{y_i\in\mathbb{R}^m,i\in\mathcal{V}}{\operatorname{minimize}}~&~\sum_{i\in\mathcal{V}}~\phi_i(y_i)\\
\operatorname{subject~to} &~ \sum_{i\in\mathcal{V}} y_i=b.
\end{split}
\end{equation}
 We say that $\mathbf{y}=[y_1^T,\ldots,y_n^T]^T\in\mathbb{R}^{nm}$ is feasible to problem \eqref{eq:probleminyz} if $\sum_{i\in\mathcal{V}} y_i=b$ and $\phi_i(y_i)<+\infty$ $\forall i\in\mathcal{V}$.
  If $\mathbf{y}$ is feasible to problem \eqref{eq:probleminyz}, then by letting $x_i$ be an optimal solution of \eqref{eq:subproblem} for all $i\in\mathcal{V}$, the point $\mathbf{x}=[x_1^T, \ldots, x_n^T]^T$ is feasible to problems \eqref{eq:prob} and \eqref{eq:barrierprob} (cf. Lemma \ref{lemma:feasibilityofy} in the appendix). The key idea we use to ensure feasibility of our algorithm is to update $x_i$ and the auxiliary variable $y_i$ in a way that guarantees feasibility of problem~\eqref{eq:probleminyz}.  
 %
 

\subsection{Distributed Resource Reallocation Algorithm (DRRA)}\label{ssec:DRRA}


We now describe our proposed algorithm in detail.
To each node $i\in\mathcal{V}$ we associate iterates $x_i^k$ and $y_i^k:=[(y_i^{\text{in},k})^T,(y_i^{\text{eq},k})^T]^T$, where $k$ is the iteration index. 
%
 Moreover, each node $i$ has access to the objective function $F_j$ and constraint matrix $A_j$ for all $j\in {\mathcal N}_i\cup \{i\}$. Requiring objective functions of neighboring nodes is not rare in distributed optimization, such as the gossip method in \cite{Lu11}.

The first step of the algorithm is initialization. The initial iterate $\mathbf{y}^0=[(y_1^0)^T, \ldots, (y_n^0)^T]^T$ can be any feasible solution to~\eqref{eq:probleminyz}, and each $x_i^0$ can be chosen as an optimal solution to problem \eqref{eq:subproblem} with $y_i=y_i^0$ for all $i\in\mathcal{V}$. In Lemma \ref{lemma:welldefinedalg} in the next section, we will show the existence of $x_i^0$ $\forall i\in\mathcal{V}$.

%
%

After the initialization, the nodes execute the following iterative process. At each iteration $k\ge 0$, a distributed algorithm is used to form an update set $U_k \subseteq {\mathcal V}$, such that
\begin{equation}\label{eq:nonconflict}
    (\mathcal{N}_i\cup\{i\})\cap(\mathcal{N}_j\cup\{j\})=\emptyset, \forall i,j\in U_k, i\ne j.
\end{equation}
Then, each $i\in U_k$ first finds $x_j^{k+1}$ $\forall j\in\mathcal{N}_i\cup\{i\}$ by solving
\begin{equation}\label{eq:xupdateprob}
		\begin{split}
			\underset{\substack{x_j\in \tilde{\mathcal{X}}_j,\\j\in\mathcal{N}_i\cup\{i\}}}{\operatorname{minimize}}~&~\sum_{j\in\mathcal{N}_i\cup\{i\}}F_j(x_j)\\
		\operatorname{subject~to}&
		\sum_{j\in\mathcal{N}_i\cup\{i\}}\!\!A_j^{\text{in}}x_j\le\sum_{j\in\mathcal{N}_i\cup\{i\}}y_j^{\text{in},k},\\
		&\sum_{j\in\mathcal{N}_i\cup\{i\}}\!\!A_j^{\text{eq}}x_j=\sum_{j\in\mathcal{N}_i\cup\{i\}}y_j^{\text{eq},k},
	\end{split}
\end{equation}
and sets $y_j^{k+1}$ $\forall j\in\mathcal{N}_i\cup\{i\}$ as
\begin{align}
	&y_j^{\text{in},k+1}\!=\!A_j^{\text{in}}x_j^{k+1}\!+\!\frac{1}{|\mathcal{N}_i|\!+\!1}\!\sum_{\ell\in\mathcal{N}_i\cup\{i\}}\!\!(y_\ell^{\text{in},k}\!-\!A_\ell^{\text{in}}x_\ell^{k+1}),\label{eq:yupdatein}\\
	& y_j^{\text{eq},k+1}=A_j^{\text{eq}}x_j^{k+1}.\label{eq:yupdateeq}
\end{align}
By treating each $y_i^k$ as the resource held by node $i$, the update at each step can be viewed as a resource reallocation between nodes $j\in \mathcal{N}_i\cup\{i\}$ for all $i\in U_k$. We can prove (Lemma~\ref{lemma:welldefinedalg} in the next section) that the optimal set of problem \eqref{eq:xupdateprob} is non-empty for any $k\ge 0$ and $i\in U_k$.

A detailed description of DRRA is given in Algorithm \ref{alg:DRRA}.
{
	\renewcommand{\baselinestretch}{1.05}
	\begin{algorithm} [!htb]
		\caption{\small Distributed Resource Reallocation Algorithm (DRRA)}
		\label{alg:DRRA}
		\begin{algorithmic}[1]
			\small
			\STATE \textbf{Initialization:}
			\STATE All the nodes cooperatively choose a feasible solution $\mathbf{y}^0=[(y_1^0)^T, \ldots, (y_n^0)^T]^T$ of problem \eqref{eq:probleminyz} (cf. Section \ref{ssec:initializationy}).
			\STATE Each node $i\in\mathcal{V}$ sets $x_i^0$ as an optimal solution of problem \eqref{eq:subproblem} with $y_i=y_i^0$.
			\FOR{$k=0,1,\ldots$}
			\STATE Form update set $U_k$ satisfying \eqref{eq:nonconflict}  (cf. Section \ref{ssec:distributedselection}).
			\FOR{each node $i\in\mathcal{V}$}
            \IF{$i\in U_k$}
            \STATE retrieve $y_j^k$ from all neighbors $j\in\mathcal{N}_i$.
            \STATE determine $x_j^{k+1}$ $\forall j\in\mathcal{N}_i\cup\{i\}$ by solving \eqref{eq:xupdateprob}.
            \STATE compute $y_j^{k+1}$ $\forall j\in\mathcal{N}_i\cup\{i\}$ according to \eqref{eq:yupdatein}--\eqref{eq:yupdateeq}.
            \STATE send $x_j^{k+1}$ and $y_j^{k+1}$ to every neighbor $j\in\mathcal{N}_i$.
            \ELSIF{$i\in\cup_{j\in U_k}\mathcal{N}_j$}
            \STATE respond to the request for $y_i^k$ from $j\in\mathcal{N}_i\cap U_k$.
            \STATE receive $x_i^{k+1}$ and $y_i^{k+1}$ from $j\in\mathcal{N}_i\cap U_k$.
            \ELSE
                \STATE set $x_i^{k+1}=x_i^k$ and $y_i^{k+1}=y_i^k$.
            \ENDIF
        \ENDFOR
			\ENDFOR
		\end{algorithmic}
	\end{algorithm}
}

All the steps in Algorithm \ref{alg:DRRA} are distributed, except for the initialization of $\mathbf{y}^0$ and the formation of $U_k$, whose distributed implementation will be discussed in Section \ref{ssec:initializationy} and Section \ref{ssec:distributedselection}, respectively. The communication at each iteration $k\ge 0$ occurs at steps 5,8,11,13,14, including the distributed formation of $U_k$ and the transmissions of $y_j^k$, $x_j^{k+1}$, and $y_j^{k+1}$ between each $i\in U_k$ and its neighbors $j\in\mathcal{N}_i$.

\subsection{Initialization of $\mathbf{y}^0$}\label{ssec:initializationy}

Finding a feasible $\mathbf{y}^0$ in a distributed fashion is usually easy. In Lemma \ref{lemma:feasibilityofy} in the appendix we will show that $\mathbf{y}^0$ is feasible to \eqref{eq:probleminyz} if the following two conditions hold: 1) $\sum_{i\in\mathcal{V}} y_i^0=b$ and 2) problem \eqref{eq:subproblem} with $y_i=y_i^0$ is feasible for all $i\in\mathcal{V}$. These two conditions can be satisfied in a distributed way for many problems. For example, for the constraint in \eqref{eq:examplecompactconstraints}, we can simply let $y_i^0=b/n\ge 0$ $\forall i\in\mathcal{V}$, where $b/n$ is either known, or can be calculated by a distributed consensus scheme (e.g., \cite{ChenJY06}).

The initialization can be further simplified if the global inequality constraint in \eqref{eq:prob} is absent. In this case, suppose each node $i\in\mathcal{V}$ knows some $\tilde{x}_i^0\in\mathbb{R}^{d_i}$ such that $\tilde{\mathbf{x}}^0=[(\tilde{x}_1^0)^T, \ldots, (\tilde{x}_n^0)^T]^T$ is feasible to problem \eqref{eq:barrierprob}, which is easier to find than $\mathbf{y}^0$ because the closed form of problem \eqref{eq:barrierprob} is known while that of each $\phi_i$ is difficult to acquire in general. Then, we can let $y_i^0=A_i\tilde{x}_i^0$ $\forall i\in\mathcal{V}$. This strategy is fully decentralized.

\subsection{Distributed Formation of the Update Set  $U_k$}\label{ssec:distributedselection}

The problem of forming an update set $U_k$ which satisfies~\eqref{eq:nonconflict} is similar to contention-resolution in wireless networks and can be solved using a number of different decentralized algorithms. Examples include 
the asynchronous Poisson clock model in \cite{Bankhamer18} and the distributed random scheduling method in \cite{Rhee09}. In addition, we propose the following voting-based selection procedure with low communication overhead.

\textbf{Voting-based selection}: Each node $i\in\mathcal{V}$ first draws a real number $v_i^k$ from the uniform distribution on the interval $[0,1]$, and then votes for the node $j\in\mathcal{N}_i\cup\{i\}$ such that $v_j^k = \min\{v_\ell^k: \ell\in\mathcal{N}_i\cup\{i\}\}$. Without loss of generality, we assume $v_j^k$, $j\in\mathcal{N}_i\cup\{i\}$ are distinct for all $i\in\mathcal{V}$. If node $i\in\mathcal{V}$ gets all the votes from $\mathcal{N}_i\cup\{i\}$, then $i\in U_k$. In this selection strategy, the communication cost only includes the broadcasting of $v_i^k$ from each $i\in\mathcal{V}$ to $j\in\mathcal{N}_i$ and the transmission of each node's voting decision.

\section{Convergence Analysis}\label{sec:convergence}
This section theoretically analyses the convergence properties of DRRA. We assume all the $x_i^k$'s and $y_i^k$'s in this section are generated by Algorithm \ref{alg:DRRA} without further mention.

\subsection{Main Result}\label{ssec:mainresult}

To present the convergence results, we make the following assumption on the selection of $U_k$.

\begin{assumption}\label{asm:coordinateselection}
For any $i\in \mathcal{V}$ and $k\ge 0$, $i\in U_k$ with a probability $p_i^k\ge\tilde{p}$ for some $\tilde{p}>0$.
\end{assumption}
The voting-based selection method introduced in Section \ref{ssec:distributedselection} satisfies Assumption \ref{asm:coordinateselection} with $\tilde{p}=1/n$.

The following theorem establishes the convergence of $x_i^k$, $i\in\mathcal{V}$ with respect to problems \eqref{eq:barrierprob} and \eqref{eq:prob}.
\begin{theorem}\label{theo:optimality}
	Suppose Assumptions \ref{asm:prob}--\ref{asm:coordinateselection} hold and that $\{ y_i^0\}_{i\in {\mathcal V}}$ is feasible to \eqref{eq:probleminyz}. Then, $\{x_i^k\}_{i\in {\mathcal V}}$ is feasible to problems \eqref{eq:barrierprob} and \eqref{eq:prob} for all $k\ge 0$. In addition,
	\begin{enumerate}
	    \item with probability $1$,
	\begin{equation*}
	    \lim_{k\rightarrow\infty} \sum_{i\in\mathcal{V}} F_i(x_i^k) = F^\star,
	\end{equation*}
	where $F^\star$ is the optimal value of problem \eqref{eq:barrierprob}.
	\item for any $\epsilon>0$, there exists $\bar{c}>0$ such that when $c\in (0,\bar{c}]$, the following holds with probability $1$,
	\begin{equation*}
	    \lim_{k\rightarrow \infty}\sum_{i\in\mathcal{V}} f_i(x_i^k)\le f^\star+\epsilon,
	\end{equation*}
	where $f^\star$ is the optimal value of problem \eqref{eq:prob}.
	\end{enumerate}
\end{theorem}
\begin{proof}
	See Section \ref{ssec:preliminary}.
\end{proof}

\subsection{Proof of Theorem \ref{theo:optimality}}\label{ssec:preliminary}

This subsection details the proof of Theorem \ref{theo:optimality}. We start by showing that problem \eqref{eq:subproblem} with $y_i=y_i^0$ and problem \eqref{eq:xupdateprob} are both solvable, so that Algorithm \ref{alg:DRRA} is well-defined.
\begin{lemma}\label{lemma:welldefinedalg}
	Suppose that Assumption \ref{asm:prob} holds and that $\{ y_i^0\}_{i\in {\mathcal V}}$ is feasible to \eqref{eq:probleminyz}. Then, the optimal set of problem \eqref{eq:subproblem} with $y_i=y_i^0$ for all $i\in\mathcal{V}$ and the optimal set  of problem \eqref{eq:xupdateprob} for all $i\in U_k$ and all $k\geq 0$ are  non-empty.
\end{lemma}
\begin{proof}
	See Appendix \ref{ssec:proofoflemmawelldefined}.
\end{proof}

Lemma \ref{lemma:nonemptysolutionset} below establishes the convergence of $\{y_i^k\}_{i\in\mathcal{V}}$ with respect to problem \eqref{eq:probleminyz}, which, together with the results in Lemma \ref{lemma:equaloptvalue}, ensures the first result in Theorem \ref{theo:optimality}.
\begin{lemma}\label{lemma:nonemptysolutionset}
	Suppose all the conditions in Lemma \ref{lemma:welldefinedalg} hold. Then
	\begin{enumerate}[(i)]
		\item for any $k\ge 0$, $\{ y_i^k \}_{i\in {\mathcal V}}$ is feasible to problem \eqref{eq:probleminyz}.
		\item the optimal set of problem \eqref{eq:probleminyz} is non-empty.
		\item if Assumption \ref{asm:coordinateselection} also holds, then with probability $1$,
		\begin{equation*}
		\lim_{k\rightarrow\infty}\sum_{i\in\mathcal{V}}\phi_i(y_i^k)=\Phi^\star,
		\end{equation*}
		where $\Phi^\star$ is the optimal value of problem \eqref{eq:probleminyz}.
	\end{enumerate} 
\end{lemma}
\begin{proof}
	See Appendix \ref{ssec:proofoflemmanonempy}.
\end{proof}


\begin{lemma}\label{lemma:equaloptvalue}
	Suppose that all the conditions in Theorem \ref{theo:optimality} hold. Then
	\begin{enumerate}[(i)]
		\item the optimal values of problems \eqref{eq:barrierprob} and \eqref{eq:probleminyz} are identical, i.e., $\Phi^\star=F^\star$.
		\item for any $k\ge 0$, $\sum_{i\in\mathcal{V}}\phi_i(y_i^k)=\sum_{i\in\mathcal{V}}F_i(x_i^k)$.
		\item for any $k\ge 0$, $\{ x_i^k\}_{i\in \mathcal{V}}$ is feasible to \eqref{eq:barrierprob} and \eqref{eq:prob}.
	\end{enumerate}
\end{lemma}
\begin{proof}
	See Appendix \ref{ssec:proofoflemmaequivalence}.
\end{proof}

It is straightforward to see that the first result in Theorem \ref{theo:optimality} follows from Lemmas \ref{lemma:nonemptysolutionset} -- \ref{lemma:equaloptvalue}.

Finally, we bound the difference between the objective error of $\{x_i^k\}_{i\in{\mathcal V}}$ with respect to problems \eqref{eq:prob} and \eqref{eq:barrierprob}.
\begin{lemma}\label{lemma:optimality}
	Suppose that all the conditions in Lemma \ref{lemma:welldefinedalg} hold. Then, for any $\epsilon>0$, there exists $\bar{c}>0$ such that when $c\in (0,\bar{c}]$
	\begin{align}
	    &\sum_{i\in\mathcal{V}} f_i(x_i^k)-f^\star\le \sum_{i\in\mathcal{V}} F_i(x_i^k)-F^\star+\epsilon\label{eq:gap}
	\end{align}
	holds for every $k\ge 0$.
\end{lemma}
\begin{proof}
	See Appendix \ref{ssec:proofofthmasymp}.
\end{proof}
By Lemma \ref{lemma:optimality} and the first result in Theorem \ref{theo:optimality}, we obtain the second result in Theorem \ref{theo:optimality}.

%
%


\section{Numerical Experiment}\label{sec:numericalexample}

\begin{figure*}
    \centering
    \begin{subfigure}[b]{0.3\textwidth}
        \centering
        \includegraphics[width=5.4cm,height=3.9cm]{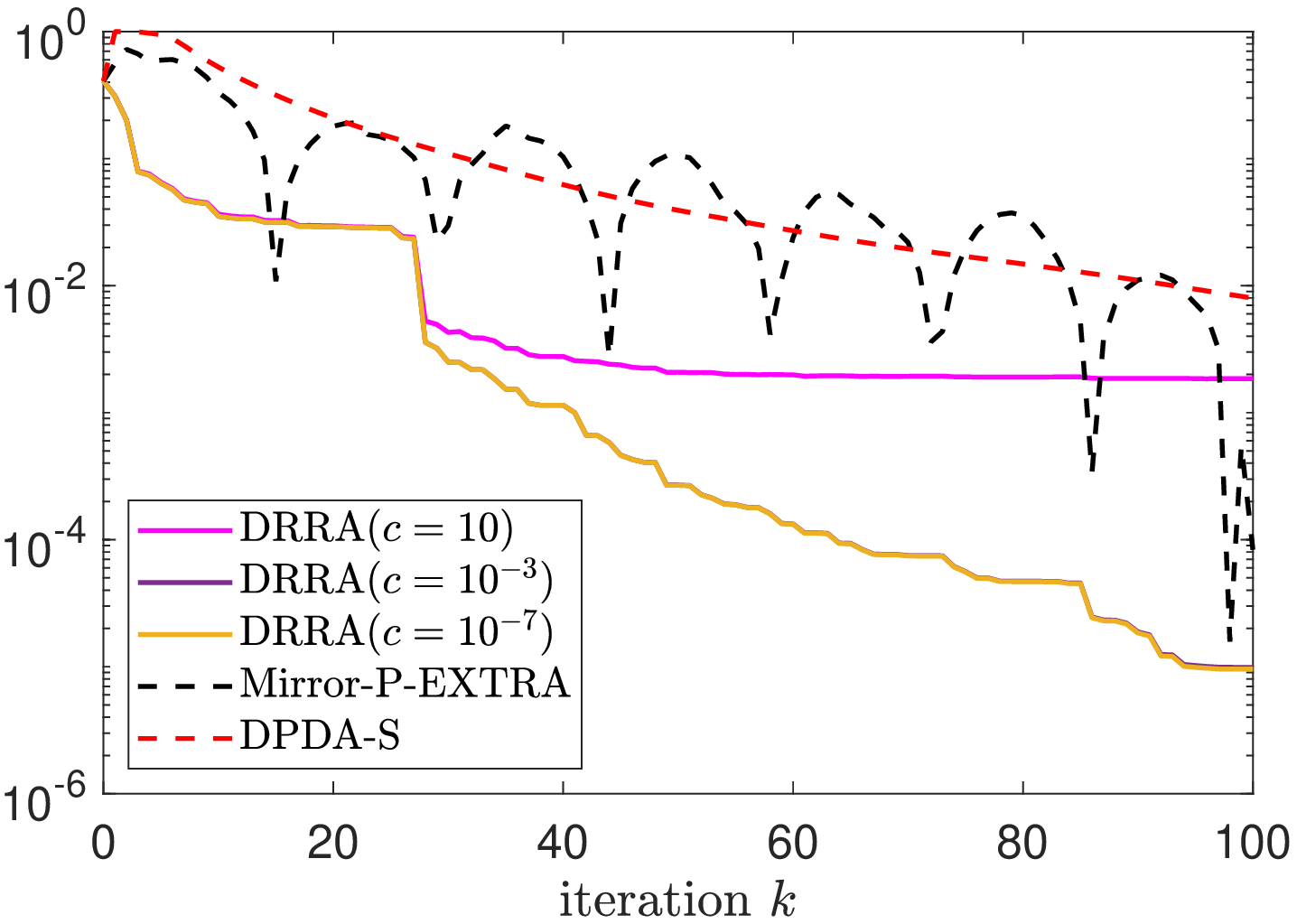}
        \caption{Objective error (Experiment I)}
        \label{fig:objvalue}
    \end{subfigure}
   \hfill
    \begin{subfigure}[b]{0.3\textwidth}
        \centering
        \includegraphics[width=5.4cm,height=3.85cm]{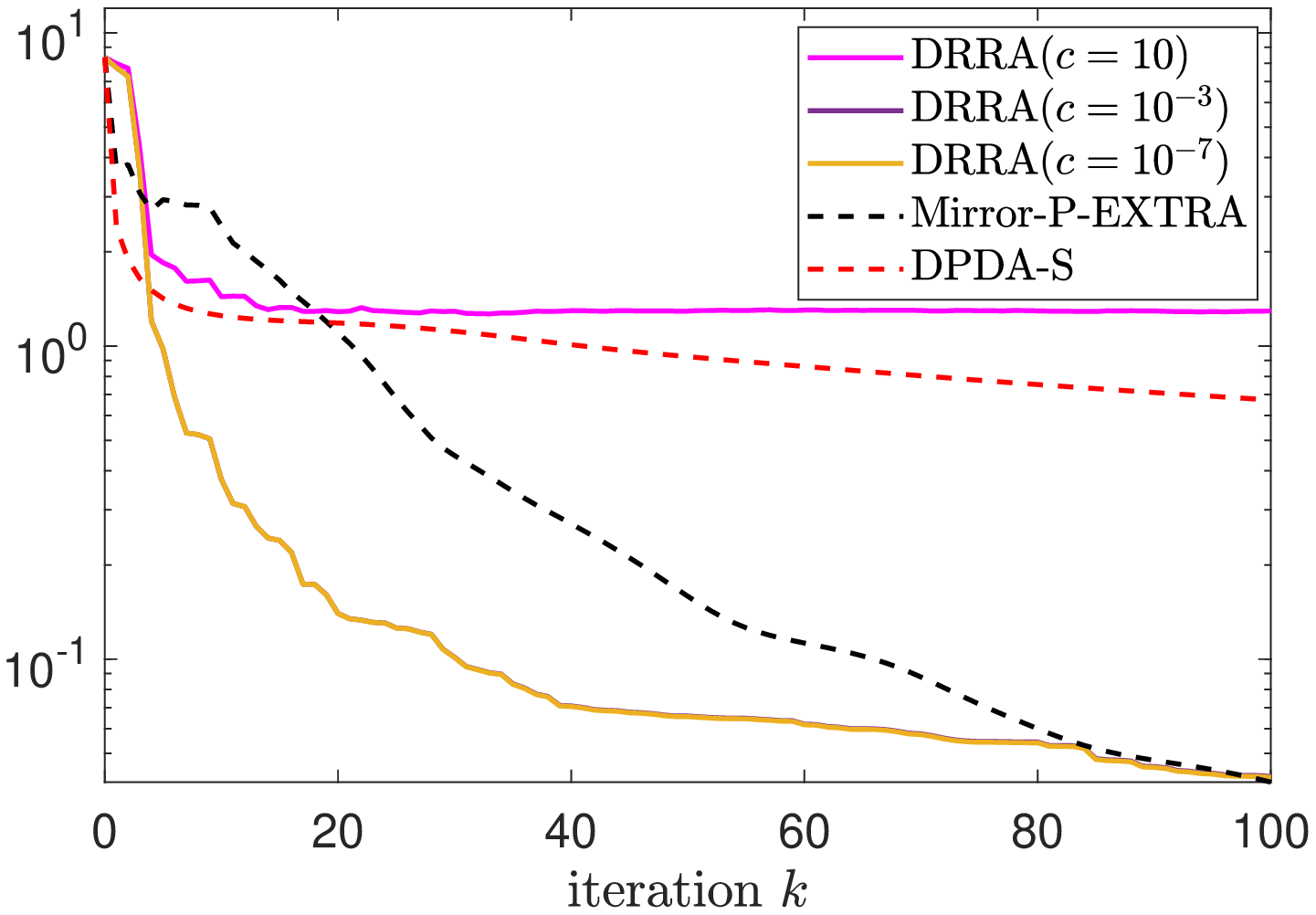}
        \caption{Objective error (Experiment II)}
        \label{fig:objvaluemultiple}
    \end{subfigure}
    \hfill
    \begin{subfigure}[b]{0.3\textwidth}
        \centering
        \includegraphics[width=5.4cm,height=3.9cm]{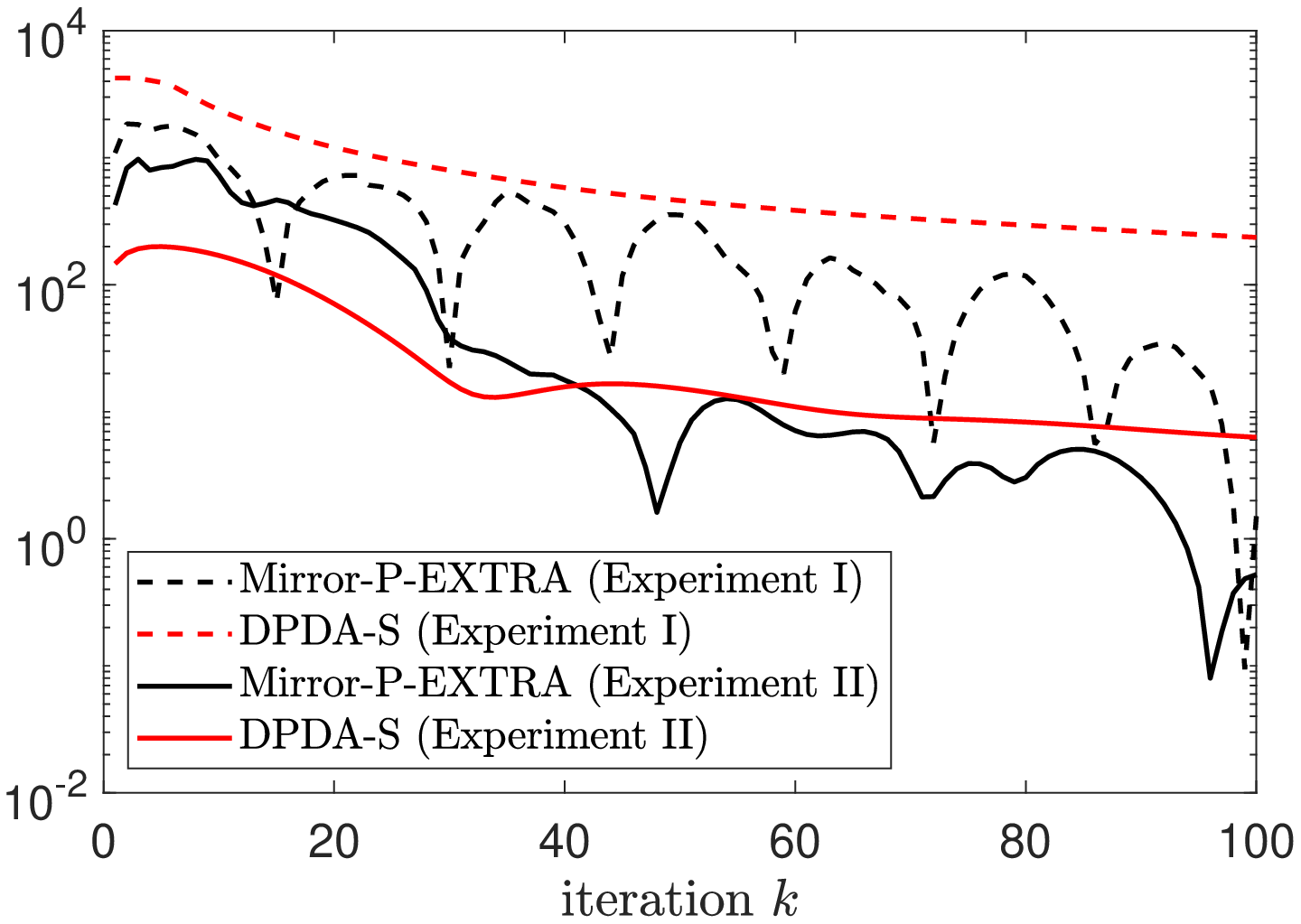}
        \caption{Feasibility error}
        \label{fig:feasi_error}
    \end{subfigure}
    \caption{
    Convergence of algorithms in Experiments I -- II}
    \label{fig:figure}
\end{figure*}

We evaluate the performance of DRRA through a numerical comparison with two alternative methods for solving the motivating examples in Section~\ref{sec:ME}. Problem data and the communication network topology are taken from the IEEE-118 bus system in Matpower \cite{Zimmerman11}. 
   
   
\textbf{Experiment I (Economic Dispatch):} We consider problem~\eqref{eq:ED} where $f_i$ is a quadratic function,  $\mathcal{X}_i:=[\ell_i, u_i]\subset\mathbb{R}$ is a closed interval, and the number of power generators is $n=54$. In addition to power generators, the IEEE-118 bus system also includes several non-generators. The generators and non-generators form an undirected, connected communication graph $\mathcal{G}'$. To form the communication network ${\mathcal G}$ for the generators, we define two power generators $i,j$ to be neighbors in $\mathcal{G}$ if there exists a path in $\mathcal{G}'$ between $i,j$ which does not include any other generators.

\textbf{Experiment II (Multiple Resources):} 
We consider problem~\eqref{eq:problemmultiple} where we focus on minimizing the disutilities of consumers in a smart grid. We let each power generator be a renewable or a coal power generator with equal probability. We treat both generators and non-generators as users and set $n=118$. The user disutility function is $f_i(x_i^{\text{renew}}, x_i^{\text{coal}})=\alpha_i(x_i^{\text{renew}}+x_i^{\text{coal}}-D_i)^2+\beta_i(x_i^{\text{coal}})^2$, where $x_i^{\text{renew}}$ and $x_i^{\text{coal}}$ are renewable and coal power consumption of node $i$, respectively, $D_i$ is the power demand of node $i$, $\alpha_i>0$ is a parameter indicating the discomfort of user $i$ for changing its demand, and $\beta_i>0$ is a parameter capturing the disapproval of user $i$ of using non-renewable energy. We set $\mathcal{X}_i=\{x_i=(x_i^{\text{renew}}, x_i^{\text{coal}})^T: x_i\ge\tilde{\ell}_i\}$, where $\tilde{\ell}_i=(-u_i,0)^T$ if $i$ is a renewable power generator and $\tilde{\ell}_i=(0,-u_i)^T$ if $i$ is a coal power generator, where $u_i$ is the same as in Experiment I, and $\tilde{\ell}_i=(0,0)^T$ if $i$ is a non-generator. The communication graph is the same $\mathcal{G}'$ as in Experiment I. Problems \eqref{eq:ED}-\eqref{eq:problemmultiple} defined by the two experiments satisfy Assumption \ref{asm:prob}.
%
%

To solve the two problems, we consider DRRA equipped with the logarithmic barrier function defined in Section~\ref{ssec:problembarriertran}, the Mirror-P-EXTRA algorithm \cite{Nedic17}, and the DPDA-S algorithm \cite{Aybat16}. The Mirror-P-EXTRA algorithm and the DPDA-S algorithm are able to produce iterates which converge asymptotically to the optimal solution of~\eqref{eq:ED} and \eqref{eq:problemmultiple}, but require every node to update at each iteration and cannot guarantee the feasibility of their iterates. We evaluate three different values for the barrier function parameter $c$ to study its effect on DRRA. The three aforementioned algorithms all involve subproblems at each iteration, which we solve using cvx~\cite{Grant14}. Moreover, for DRRA, we set $$y_i^0=\ell_i+\frac{b-\sum_{j=1}^n \ell_j}{n}\qquad \forall i\in\mathcal{V}$$ in Experiment I, and $$y_i^0=0.01(\tilde{\ell}_i-\frac{\sum_{j=1}^n \tilde{\ell}_j}{n})\qquad  \forall i\in\mathcal{V}$$ in Experiment II, which are feasible to the corresponding problem \eqref{eq:probleminyz}. For a fair comparison, we use the same initial value $x_i^0$, $i\in\mathcal{V}$ for the primal iterates of all the simulated algorithms. All the parameters in Mirror-P-EXTRA and DPDA-S are fine-tuned within their theoretically allowed ranges for the fastest convergence. We use the voting-based selection method detailed in Section \ref{ssec:distributedselection} for DRRA.

Figure \ref{fig:figure}(a)--(b) shows the convergence of the relative objective error $(\sum_{i\in\mathcal{V}} f_i(x_i^k)-f^\star)/f^\star$ in the two experiments, where $f^\star$ is the optimal value found off-line by solving the full problem using cvx. In addition, the curves for DRRA with $c=10^{-3}$ and $c=10^{-7}$ coincide in Figure \ref{fig:figure}(a)--(b). At each iteration of Mirror-P-EXTRA and DPDA-S, the iterate $x_i^k$, $i\in\mathcal{V}$ satisfies the local constraints in \eqref{eq:ED} and \eqref{eq:problemmultiple}, but it does not necessarily satisfy the global equality constraint. Thus, we plot the feasibility errors $\|\sum_{i\in\mathcal{V}} x_i^k-b\|$ for Experiment I and $\|\sum_{i\in\mathcal{V}} x_i^k\|$ for Experiment II in Fig.~\ref{fig:figure}(c). 

Figure \ref{fig:figure}(a)--(b) shows that DRRA with $c=10^{-3}$ and $c=10^{-7}$ (purple and yellow curves) converge faster than the other methods. %
Moreover, Figure \ref{fig:figure}(c) indicates that both Mirror-P-EXTRA and DPDA-S require many iterations to reach a small feasibility error, while the iterates of DRRA are always feasible. Finally, we observe that both Figure 1(a)-(b) show that DRRA with smaller $c$ converge to a higher accuracy solution, which is consistent with our intuition because a smaller $c$ implies a smaller gap between the original problem \eqref{eq:prob} and the transformed problem \eqref{eq:barrierprob}.


\section{Conclusion}\label{sec:conclusion}

We have developed a distributed resource reallocation algorithm (DRRA) for solving a class of optimal resource allocation problems. Unlike most existing distributed algorithms that only have asymptotic feasibility guarantees, every iterate of DRRA is feasible. Moreover, DRRA allows for local constraints, multiple global constraints, and multi-dimensional local decision variables, while the existing feasible methods that can handle local constraints only address problems with one global constraint and one-dimensional local variables. We also demonstrated the practical performance of DRRA via numerical experiments. 

\appendix

\subsection{Proof of Lemma~\ref{lemma:welldefinedalg}}\label{ssec:proofoflemmawelldefined}

\subsubsection{Preliminary lemmas}
We begin by introducing three lemmas which will be used in many places of the proof.
\begin{lemma}\label{lemma:solutionexistence}
	Suppose $X,Y\subseteq\mathbb{R}^d$ are convex, $f:X\rightarrow\mathbb{R}$ is convex and continuous, and $X,Y$ satisfy the following:
	\begin{enumerate}[(i)]
		\item $X$ is open and when $x$ goes to the boundary of $X$ from its interior, $f(x)$ goes to $+\infty$.
		\item $Y$ is closed.
		\item $X\cap Y$ is non-empty and bounded.
	\end{enumerate}
	Then, the optimal set of the following problem is non-empty and compact:
	\begin{equation*}
	\underset{x\in X\cap Y}{\operatorname{minimize}}~f(x).
	\end{equation*}
\end{lemma}
\begin{proof}
We refer to~\cite{Wu21a} due to space limitations.
\end{proof}

\begin{lemma}\label{lemma:lemmaertyofbarrierobj}
	Suppose Assumption \ref{asm:prob} holds. Then, both the feasible set and the optimal set of problem \eqref{eq:barrierprob} are non-empty and bounded.
\end{lemma}
\begin{proof}
By Assumption \ref{asm:prob}, the feasible set of \eqref{eq:barrierprob} is non-empty and bounded. Let $x=(x_1^T,\ldots,x_n^T)^T$, $f(x)=\sum_{i\in\mathcal{V}} F_i(x_i)$, $X=\{x:~x_i\in\tilde{\mathcal{X}}_i, \forall i\in\mathcal{V}\}$, and $Y=\{x: \sum_{i\in\mathcal{V}} A_i^{\text{in}}x_i\le b^{\text{in}},~\sum_{i\in\mathcal{V}} A_i^{\text{eq}}x_i = b^{\text{eq}}\}$. Then, since all conditions in Lemma \ref{lemma:solutionexistence} hold, the result follows.
\end{proof}

\begin{lemma}\label{lemma:feasibilityofy}
    Suppose that Assumption~\ref{asm:prob} holds. Also suppose $\{y_i\}_{i\in\mathcal{V}}$ satisfies $\sum_{i\in\mathcal{V}} y_i = b$ and the feasible set of problem \eqref{eq:subproblem} is non-empty for all $i\in\mathcal{V}$. Then, $\{\tilde{x}_i\}_{i\in\mathcal{V}}$ is feasible to problem \eqref{eq:barrierprob} if $\tilde{x}_i$ is feasible to \eqref{eq:subproblem} for all $i\in\mathcal{V}$. In addition, the optimal set of \eqref{eq:subproblem} is non-empty and bounded, and $\{y_i\}_{i\in\mathcal{V}}$ is feasible to \eqref{eq:probleminyz}.
\end{lemma}
\begin{proof}
    Since $\sum_{i\in\mathcal{V}} y_i=b$, if $\tilde{x}_i$ is feasible to problem \eqref{eq:subproblem} for all $i\in\mathcal{V}$, then $\{\tilde{x}_i\}_{i\in\mathcal{V}}$ is feasible to problem \eqref{eq:barrierprob}. In addition, by Lemma \ref{lemma:lemmaertyofbarrierobj}, the feasible set of \eqref{eq:barrierprob} is bounded. Therefore, the feasible set of \eqref{eq:subproblem} is bounded for all $i\in\mathcal{V}$. Letting $x=x_i$, $f(x)=f_i(x_i)$, $X=\tilde{\mathcal{X}}_i$, and $Y=\{x_i\in\mathbb{R}^{d_i}: A_i^{\text{in}}x_i\le 
   y_i^{\text{in}},~A_i^{\text{eq}}x_i=y_i^{\text{eq}}\}$ in Lemma \ref{lemma:solutionexistence}, we find that the optimal set of problem \eqref{eq:subproblem} is non-empty and bounded for all $i\in\mathcal{V}$. This, together with $\sum_{i\in\mathcal{V}} y_i=b$, implies the feasibility of $\{y_i\}_{i\in\mathcal{V}}$ to problem \eqref{eq:probleminyz}.
\end{proof}

\subsubsection{Main part of the proof}\label{ssec:mainproofoflemmawelldefined}

Since $\{y_i^0\}_{i\in\mathcal{V}}$ is feasible to problem \eqref{eq:probleminyz}, it follows from Lemma \ref{lemma:feasibilityofy} that the optimal set of problem \eqref{eq:subproblem} with $y_i=y_i^0$ is non-empty.

Below, we first prove that the optimal set of \eqref{eq:xupdateprob} is non-empty for all $i\in U_k$ if $\{y_j^k\}_{j\in\mathcal{V}}$ is feasible to \eqref{eq:probleminyz}. Suppose $i\in U_k$ and $\{y_j^k\}_{j\in\mathcal{V}}$ is feasible to \eqref{eq:probleminyz}. We let $x$ be the vector formed by stacking the local decision vectors $x_j$, $j\in\mathcal{N}_i\cup\{i\}$, define $f(x)=\sum_{j\in\mathcal{N}_i\cup\{i\}} F_j(x_j)$, $X=\{x: x_j\in\tilde{\mathcal{X}}_j,\forall j\in\mathcal{N}_i\cup\{i\}\}$, and let $Y$ be the constraint set formed by the linear inequality and equality constraints of \eqref{eq:xupdateprob} in Lemma \ref{lemma:solutionexistence}. Clearly, the conditions (i)--(ii) in Lemma \ref{lemma:solutionexistence} hold. If $X\cap Y$ is non-empty and bounded, then all the conditions in Lemma \ref{lemma:solutionexistence} hold and thus, the optimal set of \eqref{eq:xupdateprob} is non-empty.

The set $X\cap Y$ is non-empty because it contains $\{x_j^k\}_{j\in\mathcal{N}_i\cup\{i\}}$. Suppose $\{x_j'\}_{j\in\mathcal{N}_i\cup\{i\}}$ is an arbitrary element in $X\cap Y$ and let $x_j'=x_j^k$ for all $j\in\mathcal{V}-(\mathcal{N}_i\cup\{i\})$. Since $\sum_{j\in\mathcal{V}} y_j^k=b$, $\{x_j'\}_{j\in\mathcal{V}}$ is feasible to problem \eqref{eq:barrierprob} whose feasible set is bounded by Lemma \ref{lemma:lemmaertyofbarrierobj}. Then, all the $x_j'$'s are in a bounded set and therefore, $X\cap Y$ is bounded.

Concluding the two paragraphs above, the optimal set of \eqref{eq:xupdateprob} is non-empty when $\{y_i^k\}_{i\in\mathcal{V}}$ is feasible to \eqref{eq:probleminyz}.

Next, we prove the feasibility of $\{y_i^k\}_{i\in\mathcal{V}}$ $\forall k\ge 0$ to problem \eqref{eq:probleminyz} by induction. Suppose $\{y_i^k\}_{i\in\mathcal{V}}$ is feasible to \eqref{eq:probleminyz}, which holds at $k=0$ and implies the existence of $\{x_i^{k+1}\}_{i\in\mathcal{V}}$ and $\{y_i^{k+1}\}_{i\in\mathcal{V}}$. Because $\sum_{i\in\mathcal{N}_j\cup\{j\}} y_i^{k+1}=\sum_{i\in\mathcal{N}_j\cup\{j\}}y_i^k$ $\forall j\in U_k$, $y_i^{k+1}=y_i^k$ $\forall i\in\mathcal{V}-\cup_{j\in U_k}(\mathcal{N}_j\cup\{j\})$, and $\sum_{i\in\mathcal{V}} y_i^k=b$, we have $\sum_{i\in\mathcal{V}} y_i^{k+1}=b$. In addition, $x_i^{k+1}$ is feasible to problem \eqref{eq:subproblem} with $y_i=y_i^{k+1}$ for all $i\in\mathcal{V}$. By Lemma \ref{lemma:feasibilityofy}, $\{y_i^{k+1}\}_{i\in\mathcal{V}}$ is feasible to problem \eqref{eq:probleminyz}. Accordingly, $\{y_j^k\}_{j\in\mathcal{V}}$ is feasible to \eqref{eq:probleminyz} for all $k\ge 0$.

Concluding all the above, the optimal set of \eqref{eq:xupdateprob} is non-empty for all $k\ge 0$ and $i\in U_k$.

\subsection{Proof of Lemma \ref{lemma:nonemptysolutionset}}\label{ssec:proofoflemmanonempy}

\subsubsection{Proof of (i)}
This result has been derived in the proof of Lemma \ref{lemma:welldefinedalg} in Appendix \ref{ssec:mainproofoflemmawelldefined}).
\subsubsection{Proof of (ii)}\label{sssec:part3proofoflemmanonempty}
To show that the optimal set of \eqref{eq:probleminyz} is non-empty, we simply construct an optimal solution. To this end, we suppose that $\{x_i^\star\}_{i\in\mathcal{V}}$ is an optimal solution to problem \eqref{eq:barrierprob}, which exists due to Lemma \ref{lemma:lemmaertyofbarrierobj}, and define $y_i^\star=[(y_i^{\text{in},\star})^T, (y_i^{\text{eq},\star})^T]$ $\forall i\in\mathcal{V}$ with
\begin{equation}\label{eq:deftildey}
y_i^{\text{in},\star}=A_i^{\text{in}}x_i^\star+\frac{b^{\text{in}}-\sum_{\ell\in\mathcal{V}} A_\ell^{\text{in}}x_\ell^\star}{n},~y_i^{\text{eq},\star}=A_i^{\text{eq}}x_i^\star.
\end{equation}
We prove that $\{y_i^\star\}_{i\in\mathcal{V}}$ is optimal to \eqref{eq:probleminyz} by showing that:
\begin{enumerate}[(a)]
	\item The variable $\{y_i^\star\}_{i\in\mathcal{V}}$ is feasible to problem \eqref{eq:probleminyz} and satisfies $\phi_i(y_i^\star)\le F_i(x_i^\star)$ $\forall i\in\mathcal{V}$.
	\item For any feasible solution $\{y_i\}_{i\in\mathcal{V}}$ of problem \eqref{eq:probleminyz}, $\sum_{i\in\mathcal{V}} \phi_i(y_i)\ge \sum_{i\in\mathcal{V}} F_i(x_i^\star)$.
\end{enumerate}
If (a) and (b) hold, then $\{y_i^\star\}_{i\in\mathcal{V}}$ is feasible to problem \eqref{eq:probleminyz} and $\sum_{i\in\mathcal{V}} \phi_i(y_i^\star)\le \sum_{i\in\mathcal{V}} \phi_i(y_i)$ for any feasible solution $\{y_i\}_{i\in\mathcal{V}}$ of \eqref{eq:probleminyz}, which implies that $\{y_i^\star\}_{i\in\mathcal{V}}$ is optimal to \eqref{eq:probleminyz}.

To prove (a), note that $x_i^\star$ is feasible to problem \eqref{eq:subproblem} with $y_i=y_i^\star$ for all $i\in\mathcal{V}$. Accordingly, $\phi_i(y_i^\star)\le F_i(x_i^\star)$. In addition, $\sum_{i\in\mathcal{V}}y_i^\star=b$ due to \eqref{eq:deftildey}. Then by Lemma \ref{lemma:feasibilityofy}, $\{y_i^\star\}_{i\in\mathcal{V}}$ is feasible to problem \eqref{eq:probleminyz}.

Below, we prove (b). Since $\{y_i\}_{i\in\mathcal{V}}$ is feasible to problem \eqref{eq:probleminyz}, by Lemma \ref{lemma:feasibilityofy}, there exists an optimal solution $x_i$ to problem \eqref{eq:subproblem} for all $i\in\mathcal{V}$ and $\{x_i\}_{i\in\mathcal{V}}$ is feasible to problem \eqref{eq:barrierprob}. In addition, $\{x_i^\star\}_{i\in\mathcal{V}}$ is an optimal solution to problem \eqref{eq:barrierprob}. Then, $\sum_{i\in\mathcal{V}}\phi_i(y_i)=\sum_{i\in\mathcal{V}}F_i(x_i)\ge\sum_{i\in\mathcal{V}}F_i(x_i^\star)$.

\subsubsection{Preliminary lemmas for the proof of (iii)}\label{sssec:preproofoflemmanonempty}
 Below, we derive some lemmas to facilitate the proof of Lemma \ref{lemma:nonemptysolutionset}(iii).
\begin{lemma}\label{lemma:dualstronglyconcave}
	Suppose that Assumption \ref{asm:prob} holds and $\{y_i\}_{i\in\mathcal{V}}$ is feasible to problem \eqref{eq:probleminyz}. Then, for any $i\in\mathcal{V}$, $\phi_i$ is differentiable at $y_i$ and $\nabla \phi_i(y_i) = -u_i^\star(y_i)$, where $u_i^\star(y_i)$ is the unique geometric multiplier of \eqref{eq:subproblem}.
\end{lemma}
\begin{proof}
We refer to \cite{Wu21a} due to space limitations.
\end{proof}

\begin{lemma}\label{lemma:optimalsolutiony}
    Suppose that all the conditions in Lemma \ref{lemma:welldefinedalg} hold. Then, $\{y_j^{k+1}\}_{j\in\mathcal{N}_i\cup\{i\}}$ is an optimal solution to the following problem for all $k\ge 0$ and $i\in U_k$:
    \begin{equation}\label{eq:ykneighbouringupdate}
        \begin{split}
        \underset{y_j\in\mathbb{R}^m, j\in\mathcal{N}_{i}\cup\{i\}}{\operatorname{minimize}} &~ \sum_{j\in\mathcal{N}_i\cup\{i\}} \phi_j(y_j)\\
        \operatorname{subject~to}\quad &~ \sum_{j\in\mathcal{N}_i\cup\{i\}} (y_j-y_j^k) = \mathbf{0}.
        \end{split}
    \end{equation}
\end{lemma}
\begin{proof}
We prove the result by showing that for any $i\in U_k$, $\{y_j^{k+1}\}_{j\in \mathcal{N}_i\cup\{i\}}$ is feasible to problem \eqref{eq:ykneighbouringupdate} and for any feasible solution $\{y_j'\}_{j\in\mathcal{N}_i\cup\{i\}}$ of \eqref{eq:ykneighbouringupdate}, $\sum_{j\in\mathcal{N}_i\cup\{i\}}\phi_j(y_j')\ge \sum_{j\in\mathcal{N}_i\cup\{i\}}\phi_j(y_j^{k+1})$. This implies the optimality of $\{y_j^{k+1}\}_{j\in\mathcal{N}_i\cup\{i\}}$ to \eqref{eq:ykneighbouringupdate}.

Let $\mathcal{X}_i(y_i)=\{x_i\in\tilde{\mathcal{X}}_i: A_i^{\text{in}}x_i \le y_i^{\text{in}}, A_i^{\text{eq}}x_i = y_i^{\text{eq}}\}$ $\forall i\in\mathcal{V}$, which is the feasible set of problem \eqref{eq:subproblem}. Since $x_j^{k+1}\in\mathcal{X}_j(y_j^{k+1})$ $\forall j\in\mathcal{V}$ and $\sum_{j\in\mathcal{N}_i\cup\{i\}} (y_j^{k+1}-y_j^k)=\mathbf{0}$ $\forall i\in U_k$, $\{y_j^{k+1}\}_{j\in\mathcal{N}_i\cup\{i\}}$ is feasible to \eqref{eq:ykneighbouringupdate} for all $i\in U_k$.

Suppose $i\in U_k$ and $\{y_j'\}_{j\in\mathcal{N}_i\cup\{i\}}$ is a feasible solution of \eqref{eq:ykneighbouringupdate}. Let $y_j'=y_j^k$ $\forall j\in\mathcal{V}-(\mathcal{N}_i\cup\{i\})$. The variable $\{y_j'\}_{j\in\mathcal{V}}$ is feasible to problem \eqref{eq:probleminyz}. Then by Lemma \ref{lemma:feasibilityofy}, there exists $x_j'\in\mathcal{X}_j(y_j')$ such that $F_j(x_j') = \phi_j(y_j')$ for all $j\in\mathcal{V}$. Since $x_j'\in\mathcal{X}_j(y_j')$ $\forall j\in\mathcal{N}_i\cup\{i\}$ and $\sum_{j\in\mathcal{N}_i\cup\{i\}}y_j'=\sum_{j\in\mathcal{N}_i\cup\{i\}}y_j^k$, $\{x_j'\}_{j\in\mathcal{N}_i\cup\{i\}}$ is feasible to problem \eqref{eq:xupdateprob}. Moreover, $\{x_j^{k+1}\}_{j\in\mathcal{N}_i\cup\{i\}}$ is optimal to \eqref{eq:xupdateprob}. Accordingly,
\begin{equation}\label{eq:sumphiislarger}
    \sum_{j\in\mathcal{N}_i\cup\{i\}}\!\!\!\!\phi_j(y_j')\!=\!\!\!\sum_{j\in\mathcal{N}_i\cup\{i\}}\!\!\!\!F_j(x_j')\!\ge\!\!\! \sum_{j\in\mathcal{N}_i\cup\{i\}}\!\!\!\!F_j(x_j^{k+1}).
\end{equation}
In addition, $x_j^{k+1}\in\mathcal{X}_j(y_j^{k+1})$  $\forall j\in\mathcal{V}$ and therefore
\begin{equation}\label{eq:phismallerthansF}
    \phi_j(y_j^{k+1})\le F_j(x_j^{k+1}), \forall j\in\mathcal{V}.
\end{equation}
By \eqref{eq:sumphiislarger} and \eqref{eq:phismallerthansF}, $\sum_{j\in\mathcal{N}_i\cup\{i\}}\!\phi_j(y_j')\ge \sum_{j\in\mathcal{N}_i\cup\{i\}}\!\phi_j(y_j^{k+1})$.
\end{proof}

\begin{lemma}\label{lemma:boundedlevelsetproby}
	Suppose that Assumption \ref{asm:prob} holds. Then, the level sets of problem \eqref{eq:probleminyz} are bounded.
\end{lemma}
\begin{proof}
    Suppose $\{y_i^\star\}_{i\in\mathcal{V}}$ is an arbitrary optimal solution of problem \eqref{eq:probleminyz}. By Lemma \ref{lemma:feasibilityofy}, for each $i\in\mathcal{V}$, there exists an optimal solution $x_i^\star$ to problem \eqref{eq:subproblem} with $y_i=y_i^\star$, which, together with $\sum_{i\in\mathcal{V}} y_i^\star = b$, gives $y_i^{\text{eq},\star}=A_i^{\text{eq}}x_i^\star$ and
	\begin{equation}\label{eq:ystarinterval}
	A_i^{\text{in}}x_i^\star\!\le\! y_i^{\text{in},\star}\!=\!b^{\text{in}}-\!\!\sum_{\ell\in\mathcal{V}-\{i\}}\!\!\!y_{\ell}^{\text{in},\star}\le b^{\text{in}}\!-\!\!\sum_{\ell\in\mathcal{V}-\{i\}}\!\!\!A_\ell^{\text{in}}x_\ell^\star.
	\end{equation}
	Also by Lemma \ref{lemma:feasibilityofy}, $\{x_i^\star\}_{i\in\mathcal{V}}$ is feasible to problem \eqref{eq:barrierprob} whose feasible set is bounded by Lemma \ref{lemma:lemmaertyofbarrierobj}. This means that for any optimal solution $\{y_i^\star\}_{i\in\mathcal{V}}$ of \eqref{eq:probleminyz}, there exists $\{x_i^\star\}_{i\in\mathcal{V}}$ in the bounded feasible set of \eqref{eq:barrierprob} such that \eqref{eq:ystarinterval} holds. Therefore, the optimal set of \eqref{eq:probleminyz} is bounded. According to \cite[Corollary 8.7.1]{Rockafellar70}, all its level sets are bounded.
\end{proof}

For each $i\in\mathcal{V}$, let $y_{\mathcal{N}_i\cup\{i\}}$ be the vector formed by stacking $y_j$, $j\in\mathcal{N}_i\cup\{i\}$ and $\Psi_i(y_{\mathcal{N}_i\cup\{i\}})$ be the optimal value of the following problem:
\begin{equation*}
    \begin{split}
        \underset{z_j\in\mathbb{R}^m, j\in\mathcal{N}_{i}\cup\{i\}}{\operatorname{minimize}} &~ \sum_{j\in\mathcal{N}_i\cup\{i\}} \phi_j(z_j)\\
        \operatorname{subject~to}\quad &~ \sum_{j\in\mathcal{N}_i\cup\{i\}} (z_j-y_j) = \mathbf{0}.
    \end{split}
\end{equation*}

\begin{lemma}\label{lemma:equalgradient}
       Suppose that all the conditions in Lemma \ref{lemma:welldefinedalg} hold. For any $i\in\mathcal{V}$, if
        \begin{equation}\label{eq:identicalphi}
        \Psi_i(y_{\mathcal{N}_i\cup\{i\}}^k)=\sum_{j\in\mathcal{N}_i\cup\{i\}}\phi_j(y_j^k),
        \end{equation}
        then $\nabla \phi_j(y_j^k)=\nabla\phi_\ell(y_\ell^k)$ $\forall j,\ell\in\mathcal{N}_i\cup\{i\}$.
\end{lemma}
\begin{proof}
    We first show that for any $i\in\mathcal{V}$ and $k\ge 0$, $x_i^k$ is optimal to problem \eqref{eq:subproblem} with $y_i=y_i^k$. Combining \eqref{eq:sumphiislarger} with $y_j'=y_j^{k+1}$ $\forall j\in\mathcal{N}_i\cup\{i\}$ and \eqref{eq:phismallerthansF} gives
\begin{equation}\label{eq:sumyequaltosumF}
    \phi_j(y_j^{k+1})=F_j(x_j^{k+1}), \forall i\in U_k,\forall j\in\mathcal{N}_i\cup\{i\}.
\end{equation}
Because $x_i^{k+1}$ is feasible to problem \eqref{eq:subproblem} with $y_i=y_i^{k+1}$ for all $i\in\mathcal{V}$ and because of \eqref{eq:sumyequaltosumF}, $x_i^{k+1}$ is optimal to problem \eqref{eq:subproblem} with $y_i=y_i^{k+1}$ for all $i\in\cup_{j\in U_k} (\mathcal{N}_j\cup\{j\})$. In addition, $x_i^0$ is optimal to problem \eqref{eq:subproblem} with $y_i=y_i^0$ for all $i\in\mathcal{V}$ and $x_i^{k+1}=x_i^k$ and $y_i^{k+1}=y_i^k$ for all $i\in\mathcal{V}-\cup_{j\in U_k} (\mathcal{N}_j\cup\{j\})$. As a result, $x_i^k$ is optimal to problem \eqref{eq:subproblem} with $y_i=y_i^k$ for all $i\in\mathcal{V}$ and $k\ge 0$ and therefore
\begin{equation}\label{eq:identicalFandphi}
    \phi_i(y_i^k) = F_i(x_i^k), \forall i\in\mathcal{V}, \forall k\ge 0.
\end{equation}
    
Since each $\tilde{\mathcal{X}}_i$ is open, by the first-order optimality condition of \eqref{eq:subproblem} with $y_i=y_i^k$ we have $\nabla F_i(x_i^k)+A_i^Tu_i^\star(y_i^k) = \mathbf{0}, \forall i\in\mathcal{V}$, where $u_i^\star(y_i^k)$ is a geometric multiplier of \eqref{eq:subproblem} with $y_i=y_i^k$. Also, by Lemma \ref{lemma:nonemptysolutionset}(i) and Lemma \ref{lemma:dualstronglyconcave}, $\nabla \phi_i(y_i^k)=-u_i^\star(y_i^k)$. Therefore,
\begin{equation}\label{eq:graFanduik1}
    \nabla F_i(x_i^k)-A_i^T\nabla \phi_i(y_i^k) = \mathbf{0}, \forall i\in\mathcal{V}.
\end{equation}

Finally, based on \eqref{eq:graFanduik1} we prove that when \eqref{eq:identicalphi} holds, $\nabla \phi_j(y_j^k)=\nabla\phi_\ell(y_\ell^k)$ $\forall j,\ell\in\mathcal{N}_i\cup\{i\}$. For any feasible solution $\{\tilde{x}_j\}_{j\in\mathcal{N}_i\cup\{i\}}$ of problem \eqref{eq:xupdateprob}, we define $\tilde{y}_j$ $\forall j\in\mathcal{N}_i\cup\{i\}$ as $\tilde{y}_j^{\text{in}}=A_j^{\text{in}}\tilde{x}_j+\frac{1}{|\mathcal{N}_i|+1}\sum_{\ell\in\mathcal{N}_i\cup\{i\}}(y_\ell^{\text{in},k}-A_\ell^{\text{in}}\tilde{x}_\ell)$ and $\tilde{y}_j^{\text{eq}}=A_j^{\text{eq}}\tilde{x}_j$. Since $\tilde{x}_j\in\mathcal{X}_j(\tilde{y}_j)$, where $\mathcal{X}_j(\tilde{y}_j)$ is defined in the proof of Lemma \ref{lemma:optimalsolutiony},
\begin{equation}\label{eq:phiissmallerthanF}
    \phi_j(\tilde{y}_j)\le F_j(\tilde{x}_j), \forall j\in\mathcal{N}_i\cup\{i\}.
\end{equation}
In addition, $\{\tilde{y}_j\}_{j\in\mathcal{N}_i\cup\{i\}}$ is feasible to \eqref{eq:ykneighbouringupdate} and by \eqref{eq:identicalphi}, $\{y_j^k\}_{j\in\mathcal{N}_i\cup\{i\}}$ is optimal to \eqref{eq:ykneighbouringupdate}. These, together with \eqref{eq:identicalFandphi} and \eqref{eq:phiissmallerthanF}, yield $\sum_{j\in\mathcal{N}_i\cup\{i\}} F_j(x_j^k)=\sum_{j\in\mathcal{N}_i\cup\{i\}} \phi_j(y_j^k)\le \sum_{j\in\mathcal{N}_i\cup\{i\}}\phi_j(\tilde{y}_j)\le \sum_{j\in\mathcal{N}_i\cup\{i\}}F_j(\tilde{x}_j)$, which implies that $\{x_j^k\}_{j\in\mathcal{N}_i\cup\{i\}}$ is optimal to \eqref{eq:xupdateprob}. According to the first-order optimality condition of \eqref{eq:xupdateprob} and since $\tilde{\mathcal{X}}_j$ $\forall j\in\mathcal{N}_i\cup\{i\}$ are open, there exists $u'\in\mathbb{R}^m$ such that $\nabla F_j(x_j^k)+A_j^Tu' = \mathbf{0}~ \forall j\in\mathcal{N}_i\cup\{i\}$. This, together with \eqref{eq:graFanduik1} and the full row rank property of $A_j$, yields $\nabla \phi_j(y_j^k)=u'$ $\forall j\in\mathcal{N}_i\cup\{i\}$.
\end{proof}
\begin{lemma}\label{lemma:psicontinuity}
    Suppose that Assumption \ref{asm:prob} holds. Then, $\Psi_i(y_{\mathcal{N}_i\cup\{i\}})$ is continuous on its domain for all $i\in\mathcal{V}$.
\end{lemma}
\begin{proof}
    We refer to \cite{Wu21a} due to space limitations.
\end{proof}

Let $r_i(y_{\mathcal{N}_i\cup\{i\}}\!)\!=\!\!(\sum_{j\in\mathcal{N}_i\cup\{i\}}\!\phi_j(y_j))\!-\!\Psi_i(y_{\mathcal{N}_i\cup\{i\}}\!) \forall i\in\mathcal{V}$.
\begin{lemma}\label{lemma:optconditionofproby}
	Suppose that all the conditions in Lemma \ref{lemma:welldefinedalg} hold. If $r_i(y_{\mathcal{N}_i\cup\{i\}}^k)=0$
	$\forall i\in\mathcal{V}$, then $\{y_i^k\}_{i\in\mathcal{V}}$ is optimal to problem \eqref{eq:probleminyz}.
\end{lemma}
\begin{proof}
By Lemma \ref{lemma:equalgradient}, $r_i(y_{\mathcal{N}_i\cup\{i\}}^k)=0$ means $\nabla \phi_{j}(y_{j}^k) = \nabla \phi_{\ell}(y_{\ell}^k)~\forall j,\ell\in\mathcal{N}_i\cup\{i\}$. Also, the graph $\mathcal{G}$ is connected. Therefore, $\nabla \phi_i(y_i^k)=\nabla \phi_{j}(y_{j}^k)~\forall i,j\in\mathcal{V}$. Combining this and the feasibility of $\{y_i^k\}_{i\in\mathcal{V}}$ to \eqref{eq:probleminyz} guaranteed by Lemma \ref{lemma:nonemptysolutionset}(i) yields the result.
\end{proof}

\subsubsection{Proof of (iii)}
We first prove that with probability $1$, $\lim\limits_{k\rightarrow \infty} \sum_{i\in\mathcal{V}} r_i(y_{\mathcal{N}_i\cup\{i\}}^k) = 0$. Let $\mathcal{F}^k=\{U_0,U_1,\ldots,U_{k-1}\}$ represent the selection of $U_t$ during $0\le t\le k-1$. By Assumption \ref{asm:coordinateselection} and Lemma \ref{lemma:optimalsolutiony}, $E\left[\sum_{i\in\mathcal{V}}\phi_i(y_i^{k+1})|\mathcal{F}^k\right]=\sum_{i\in\mathcal{V}}\phi_i(y_i^k)-E[\sum_{i\in U_k}r_i(y_{\mathcal{N}_i\cup\{i\}}^k)|\mathcal{F}^k]=\sum_{i\in\mathcal{V}}\phi_i(y_i^k)-\sum_{i\in\mathcal{V}} p_i^kr_i(y_{\mathcal{N}_i\cup\{i\}}^k)\le \sum_{i\in\mathcal{V}}\phi_i(y_i^k)-\sum_{i\in\mathcal{V}} \tilde{p}r_i(y_{\mathcal{N}_i\cup\{i\}}^k)$, which implies $E\left[\sum_{i\in\mathcal{V}}\phi_i(y_i^{k+1})\!-\!\Phi^\star|\mathcal{F}^k\right]\le\sum_{i\in\mathcal{V}}\phi_i(y_i^k)$ $-\Phi^\star-\sum_{i\in\mathcal{V}} \tilde{p}r_i(y_{\mathcal{N}_i\cup\{i\}}^k)$. Since $\{y_i^k\}_{i\in\mathcal{V}}$ is feasible to problem \eqref{eq:probleminyz} by Lemma \ref{lemma:nonemptysolutionset}(i), $\sum_{i\in\mathcal{V}}\phi_i(y_i^k)-\Phi^\star$ is non-negative. Also, $\sum_{i\in\mathcal{V}} \tilde{p}r_i(y_{\mathcal{N}_i\cup\{i\}}^k)$ is non-negative. Then, by the super-martingale convergence theorem \cite[Lemma 2.2]{Bertsekas11}, $\lim_{k\rightarrow \infty} \sum_{i\in\mathcal{V}} r_i(y_{\mathcal{N}_i\cup\{i\}}^k) = 0$ with probability $1$.

Next, we derive the convergence of $\sum_{i\in\mathcal{V}} \phi_i(y_i^k)$, provided that $\lim_{k\rightarrow \infty} \sum_{i\in\mathcal{V}} r_i(y_{\mathcal{N}_i\cup\{i\}}^k) = 0$ which occurs with probability $1$ by the above derivation. Because of Lemma \ref{lemma:optimalsolutiony} and $y_i^{k+1}=y_i^k$ $\forall i\in \cup_{j\in U_k}(\mathcal{N}_j\cup\{j\})$, $\sum_{i\in\mathcal{V}} \phi_i(y_i^k)$ is monotonically non-increasing and therefore, $\sum_{i\in\mathcal{V}} \phi_i(y_i^k)\le \sum_{i\in\mathcal{V}} \phi_i(y_i^0)$. In addition, $\{y_i^k\}_{i\in\mathcal{V}}$ is feasible to problem \eqref{eq:probleminyz}. Then, by Lemma \ref{lemma:boundedlevelsetproby}, the sequence $\{y_i^k\}_{i\in\mathcal{V}}$, $k\ge 0$ is bounded. By Bolzano-Weierstrass theorem, there exists a convergent sequence $\{y_i^k\}_{i\in\mathcal{V}}$, $k\in\mathcal{K}$, where $\mathcal{K}$ is a subset of the set of non-negative integers. Let $\{\tilde{y}_i\}_{i\in\mathcal{V}}$ be the limit of the sequence $\{y_i^k\}_{i\in\mathcal{V}}$, $k\in\mathcal{K}$. Because each $r_i(y_{\mathcal{N}_i\cup\{i\}}^k)$ is a continuous function on $\{y_i^k\}_{i\in\mathcal{N}_i\cup\{i\}}$ according to Lemma \ref{lemma:psicontinuity} and $\lim_{k\rightarrow \infty} \sum_{i\in\mathcal{V}} r_i(y_{\mathcal{N}_i\cup\{i\}}^k)=0$, if we let $y_i^k=\tilde{y}_i$ $\forall i\in\mathcal{V}$, then $r_i(y_{\mathcal{N}_i\cup\{i\}}^k)=0$ $\forall i\in\mathcal{V}$. This implies $\sum_{i\in\mathcal{V}} \phi_i(\tilde{y}_i)=\Phi^\star$ by Lemma \ref{lemma:optconditionofproby}. In addition, $\sum_{i\in\mathcal{V}} \phi_i(y_i^k)$ is monotonically non-increasing. Therefore, $\lim_{k\rightarrow \infty} \sum_{i\in\mathcal{V}} \phi_i(y_i^k) = \Phi^\star$.

To conclude, $\lim\limits_{k\rightarrow \infty} \sum_{i\in\mathcal{V}} \phi_i(y_i^k) = \Phi^\star$ with probability $1$.

\subsection{Proof of Lemma \ref{lemma:equaloptvalue}}\label{ssec:proofoflemmaequivalence}

Result (i) can be derived from assertions (a) and (b) in Appendix \ref{sssec:part3proofoflemmanonempty}) with $y_i=y_i^\star$, while (ii) follows from \eqref{eq:identicalFandphi}.
To prove result (iii), note that each $x_i^k$ is feasible to problem \eqref{eq:subproblem} with $y_i=y_i^k$. In addition, by Lemma \ref{lemma:nonemptysolutionset} and Lemma \ref{lemma:feasibilityofy}, $\{x_i^k\}_{i\in\mathcal{V}}$ is feasible to problem \eqref{eq:barrierprob}. It is straightforward to see that every feasible solution of \eqref{eq:barrierprob} is also feasible to \eqref{eq:prob}.

\subsection{Proof of Lemma \ref{lemma:optimality}}\label{ssec:proofofthmasymp}

We first provide an upper bound for $F^\star$. Since the feasible set of problem \eqref{eq:prob} is bounded and each $f_i$ is continuous, there exists a compact set $\mathcal{X}_\epsilon$ that satisfies: (a) $\mathcal{X}_\epsilon$ is a subset of the feasible set of problem \eqref{eq:barrierprob}; and (b) there exists $\mathbf{x}'=[(x_1')^T, \ldots, (x_n')^T]^T\in \mathcal{X}_\epsilon$ such that $\sum_{i\in\mathcal{V}} f_i(x_i')\le f^\star+\epsilon/2$. Since $\mathcal{X}_\epsilon$ is compact and each $B_i^j$, $i\in\mathcal{V}$, $j=1,\ldots,p_i$ is continuous, there exists an upper bound $\bar{B}>0$ of $\sum_{i\in\mathcal{V}}\sum_{j=1}^{p_i} B_i^j(x_i)$ for all $\mathbf{x}=[(x_1)^T, \ldots, (x_n)^T]^T\in\mathcal{X}_{\epsilon}$. Therefore, $\sum_{i\in\mathcal{V}} F_i(x_i')=\sum_{i\in\mathcal{V}} (f_i(x_i')+c\sum_{j=1}^{p_i} B_i^j(x_i'))\le \sum_{i\in\mathcal{V}} f_i(x_i')+c\bar{B}\le f^\star+\epsilon/2+c\bar{B}$. In particular, 
$F^\star\le f^\star+c\bar{B}+\epsilon/2$.

Next, we derive a lower bound for $\sum_{i\in\mathcal{V}} F_i(x_i^k)$. Since the feasible set of problem \eqref{eq:barrierprob} is bounded due to Lemma \ref{lemma:lemmaertyofbarrierobj} and $B_i^j(x_i)$ is continuous and goes to $+\infty$ when $g_i^j(x_i)$ goes to $0$ from negative values, $\sum_{i\in\mathcal{V}} \sum_{j=1}^{p_i} B_i^j(x_i)$ is lower bounded by some $\underline{B}$ in the feasible set of problem \eqref{eq:barrierprob}. Thus, $\sum_{i\in\mathcal{V}} F_i(x_i^k)=\sum_{i\in\mathcal{V}} (f_i(x_i^k)+c\sum_{j=1}^{p_i} B_i^j(x_i^k))\ge \sum_{i\in\mathcal{V}} f_i(x_i^k)+c\underline{B}$.

The desired result now follows with $\bar{c}=\epsilon/(2\bar{B}-2\underline{B})$ and $0<c\le \bar{c}$ in the upper and lower bound derived above.
\bibliographystyle{IEEEtran}
\bibliography{reference}

\end{document}